\documentclass[10pt,a4paper,reqno]{amsart}

\usepackage{amsmath,amssymb,amsthm,amsfonts}
\usepackage{amsbsy}
\usepackage[utf8]{inputenc}
\usepackage{indentfirst}
\usepackage{geometry}
\usepackage{enumitem}
\usepackage{a4wide}
\usepackage[english]{babel}
\usepackage{longtable}
\usepackage{url}
\usepackage[all]{xy}
\usepackage{subcaption}
\usepackage{aliascnt}
\usepackage{ytableau}

\usepackage{verbatim}

\usepackage{tikz}
\usetikzlibrary{matrix}
\usetikzlibrary{arrows}
\usetikzlibrary{fit}
\usetikzlibrary{shapes}

\usepackage{booktabs}

\definecolor{LinkColor}{rgb}{0,0,0} 
\definecolor{Schweinchenrosa}{HTML}{FFAAAA}
\usepackage[colorlinks=true,linkcolor=LinkColor,citecolor=LinkColor,urlcolor= LinkColor, naturalnames, hyperindex, pdfstartview=FitH, bookmarksnumbered, plainpages]{hyperref} 

\usepackage[nameinlink, capitalise]{cleveref} 

\newtheorem{maintheorem}{Theorem}
\crefname{maintheorem}{Theorem}{Theorems}

\newtheorem{maincorollary}[maintheorem]{Corollary}
\crefname{maincorollary}{Corollary}{Corollaries}

\newtheorem{theorem}{Theorem}
\crefname{theorem}{Theorem}{Theorems}

\newaliascnt{corollary}{theorem}
\newtheorem{corollary}[corollary]{Corollary}
\aliascntresetthe{corollary}
\crefname{corollary}{Corollary}{Corollaries}

\newaliascnt{lemma}{theorem}
\newtheorem{lemma}[lemma]{Lemma}
\aliascntresetthe{lemma}
\crefname{lemma}{Lemma}{Lemmata}

\newaliascnt{proposition}{theorem}
\newtheorem{proposition}[proposition]{Proposition}
\aliascntresetthe{proposition}
\crefname{proposition}{Proposition}{Propositions}

\newtheorem*{zc*}{Zassenhaus Conjecture (ZC)}
\newtheorem*{pq*}{Prime Graph Question (PQ)}
\newtheorem*{theorem*}{Theorem}

\theoremstyle{definition}

\newaliascnt{definition}{theorem}
\newtheorem{definition}[definition]{Definition}
\aliascntresetthe{definition}
\crefname{definition}{Definition}{Definitions}

\theoremstyle{remark}

\newaliascnt{example}{theorem}

\aliascntresetthe{example}
\crefname{example}{Example}{Examples}

\newaliascnt{remark}{theorem}
\newtheorem{remark}[remark]{Remark}
\aliascntresetthe{remark}
\crefname{remark}{Remark}{Remarks}

\newaliascnt{notation}{theorem}

\aliascntresetthe{notation}
\crefname{notation}{Notation}{Notations}

\newcommand{\Sz}{\operatorname{Sz}}
\newcommand{\PSL}{\operatorname{PSL}}
\newcommand{\PGL}{\operatorname{PGL}}
\newcommand{\PSU}{\operatorname{PSU}}
\newcommand{\PSp}{\operatorname{PSp}}

\newcommand{\Aut}{\operatorname{Aut}}

\newcommand{\V}{\textup{V}}
\newcommand{\ZZ}{\mathbb{Z}}

\newcommand{\cc}[1]{\ensuremath{{\texttt{#1}}}}
\newlength{\heightofhw}
\newcommand{\eigbox}[2]{\ensuremath{\textstyle{#1} \times \boxed{\rule{0cm}{\heightofhw} #2}}}

\title[Prime Graph Question for Integral Group Rings of 4-primary groups II]{On the Prime Graph Question for Integral Group Rings\\ of 4-primary groups II}
\author{Andreas B\"achle}
\address{Vakgroep Wiskunde, Vrije Universiteit Brussel, Pleinlaan 2, 1050 Brussels, Belgium}
\email{\href{mailto:abachle@vub.ac.be}{abachle@vub.ac.be}}
\author{Leo Margolis}
\address{Departamento de matem\'aticas, Facultad de matem\'aticas, Universidad de Murcia, 30100 Murcia, Spain}
\email{\href{mailto:leo.margolis@um.es}{leo.margolis@um.es}}
\thanks{The first author is supported by the Research Foundation Flanders (FWO - Vlaanderen). The second by a Marie-Curie Individual Fellowship from European Commission H2020 program.}

\begin{document}

\begin{abstract}
 In this article the study of the Prime Graph Question for the integral group ring of almost simple groups which have an order divisible by exactly $4$ different primes is continued. We provide more details on the recently developed ``lattice method'' which involves the calculation of Littlewood-Richardson coefficients. We apply the method obtaining results complementary to those previously obtained using the HeLP-method. In particular the ``lattice method'' is applied to infinite series of groups for the first time. We also prove the Zassenhaus Conjecture for four more simple groups. Furthermore we show that the Prime Graph Question has a positive answer around the vertex $3$ provided the Sylow $3$-subgroup is of order $3$.
\end{abstract}

\maketitle

\section{Introduction and Main Result}

Since the PhD thesis of G. Higman \cite{HigmanThesis} the units of the integral group ring $\mathbb{Z}G$ of a finite group $G$ have been studied intensively by many authors. See also \cite{Sandling} for details on Higman's thesis. Though wide reaching results have been obtained for some special classes of groups, e.g. Weiss' celebrated theorems for nilpotent groups \cite{Weiss91}, in general many questions remain open. One of the main open problems concerning the torsion units of $\mathbb{Z}G$ is the so called\\

\textbf{Zassenhaus Conjecture (ZC).}\ For every unit $u \in \mathbb{Z}G$ of finite order there exists a unit $x$ in the rational group algebra $\mathbb{Q}G$ such that $x^{-1} u x = \pm g$ for some $g \in G$.\\

The Zassenhaus Conjecture is known for some series of solvable groups, e.g. for nilpotent groups \cite{Weiss91}, groups with a normal Sylow $p$-subgroup complemented by an abelian group \cite{HertweckColloq} or cyclic-by-abelian groups \cite{CyclicByAbelian}. For non-solvable groups however very little is known. In particular for simple groups the conjecture is only known for a few groups, all of them isomorphic to some small $\PSL(2,q)$, see the references in the proof of Theorem \ref{ZCPSL}.

A unit $u \in \mathbb{Z}G$ is called normalized, if the coefficient sum of $u$ equals $1$ and $\mathrm{V}(\mathbb{Z}G)$ denotes the set of all normalized units of $\mathbb{Z}G$. Clearly the units of $\mathbb{Z}G$ are $\pm \mathrm{V}(\mathbb{Z}G)$ and it thus suffices to study normalized units. W. Kimmerle asked the following question which can be seen as a first approximation for the Zassenhaus Conjecture \cite{KimmiPQ}.\\

\textbf{Prime Graph Question (PQ).}\ If $\V(\ZZ G)$ contains a unit of order $pq$, does $G$ have an element of order $pq$, where $p$ and $q$ are different primes? \\

This is equivalent to say that the prime graphs of $G$ and $\mathrm{V}(\mathbb{Z}G)$ coincide. The Prime Graph Question seems especially approachable since W. Kimmerle and A. Konovalov proved a strong reduction: The Prime Graph Question has a positive answer for some group $G$ if and only if it has a positive answer for all almost simple images of $G$ \cite[Theorem 2.1]{KonovalovKimmiComp}. A group is called almost simple if it is sandwiched between a non-abelian simple group $S$ and the automorphism group of $S$. In this case $S$ is called the socle of $G$. 

In the first part of this investigation we showed exactly how much can be achieved in the study of the Prime Graph Question for groups whose order has exactly four different prime divisors using only the well known HeLP method. Such groups are called $4$-primary. In this paper we continue to study the Prime Graph Question for $4$-primary groups using a method recently developed by the authors \cite{Gitter}. In total there are $37$ specific and $3$ series of potentially infinite simple $4$-primary groups giving rise to $123$ specific and $5$ possibly infinite series of almost simple groups. For $12$ specific groups and $1$ series a positive answer to the Prime Graph Question is available in the literature (see references in \cite{HeLPTeil}). By the results of this paper and its predecessor the Prime Graph Question has a positive answer for $109$ specific and $3$ series of groups.

We also establish the Zassenhaus Conjecture for four more simple groups increasing the number of non-abelian simple groups for which the Zassenhaus Conjecture is known to thirteen.
Moreover we show that the prime graphs of $G$ and $\mathrm{V}(\mathbb{Z}G)$ coincide around the vertex $3$ assuming the Sylow $3$-subgroup of $G$ has order $3$. 
The main results are the following.

\begin{maintheorem}\label{MainTheorem}
Let $G$ be an almost simple $4$-primary group. The following table shows, whether an answer to the Prime Graph Question for $G$ is known. A simple group $S$ appearing in bold letters indicates all almost simple groups having $S$ as a socle. If $G$ appears in the left column of the table the Prime Graph Question has an affirmative answer for $G$. If $G$ is in the right column it is not known whether $\V(\ZZ G)$ contains elements of order $pq$ for the values indicated in italic in parentheses. 

\begin{longtable}{|p{.54\textwidth} | p{.4\textwidth}|} \hline
{\bf (PQ)} has affirmative answer & {\bf (PQ)} not known \\ \hline
$\boldsymbol{\PSL(2,p)}$, $p$ a prime & $\boldsymbol{\PSL(2,3^f)}$, $f \geq 7$ \textit{(6)} \\
$\boldsymbol{\PSL(2,2^f)}$, $f \geq 4$ &  \\ \hline 
$\boldsymbol{A_7}$, $\boldsymbol{A_8}$, $\boldsymbol{A_9}$, $\boldsymbol{A_{10}}$ & {} \\
$\boldsymbol{\PSL(2,25)}$, $\boldsymbol{\PSL(2,49)}$, $\boldsymbol{\PSL(3, 4)}$,  & $\boldsymbol{\PSL(3,17)}$ \textit{(51)} \\
$\boldsymbol{\PSL(3,5)}$, $\boldsymbol{\PSL(4,3)}$ & \\
$\boldsymbol{\PSU(3,4)}$, $\boldsymbol{\PSU(3,5)}$, $\boldsymbol{\PSU(3,7)}$, $\boldsymbol{\PSU(3,8)}$ & \\
$\boldsymbol{\PSU(3,9)}$, $\boldsymbol{\PSU(4,3)}$, $\boldsymbol{\PSU(4,4)}$, $\boldsymbol{\PSU(4,5)}$ & {} \\
$\boldsymbol{\PSU(5,2)}$ & { } \\
$\boldsymbol{\PSp(4,4)}$, $\boldsymbol{\PSp(4,5)}$, $\boldsymbol{\PSp(4,9)}$, $\boldsymbol{\PSp(6,2)}$ & $\boldsymbol{\PSp(4,7)}$ \textit{(35)} \\
$\boldsymbol{\operatorname{P}\Omega_+(8,2)}$, $\boldsymbol{\Sz(8)}$, $\boldsymbol{G_2(3)}$, $\boldsymbol{{}^3D_4(2)}$  & {} \\
$\boldsymbol{{}^2F_4(2)'}$, $\boldsymbol{M_{11}}$, $\boldsymbol{M_{12}}$, $\boldsymbol{J_2}$ & {} \\

$\PSL(2,27).3$, $\PSL(2,27).6$ &  $\PSL(2,27)$ \textit{(6)}, $\PSL(2,27).2$ \textit{(6)} \\
$\PSL(2,81).2a$, $\PSL(2,81).4a$, $\PSL(2,81).2^2$ &  $\PSL(2,81)$ \textit{(6)}, $\PSL(2,81).2b$ \textit{(6)} \\
$\PSL(2,81).4b$, $\PSL(2,81).(2\times 4)$ &  $\PSL(2,81).2c$ \textit{(6)},  \\
{} & $\PSL(2,243)$ \textit{(6)}, $\PGL(2,243)$ \textit{(6, 33)} \\
$\PSL(3,7)$, $\PSL(3,7).3$, $\PSL(3,7).S_3$ & $\PSL(3,7).2$ \textit{(21)} \\
$\PSL(3,8).2$, $\PSL(3,8).3$, $\PSL(3,8).6$ & $\PSL(3,8)$ \textit{(6)} \\
$\Sz(32).5$ & $\Sz(32)$ \textit{(10)} \\
\hline 
\end{longtable}
\end{maintheorem}

\begin{proof}
The theorem follows from \cite[Theorem B]{HeLPTeil} and Corollaries \ref{PQPSL}, \ref{PQPSL35}, \ref{PQPSL281}, and Lemmas \ref{PQPSL34} and \ref{PQPSL37}.
\end{proof}

Together with the reduction of Kimmerle and Konovalov mentioned above and \cite[Corollary~1.3]{Gitter} this implies

\begin{maincorollary}
Let $G$ be a finite group and assume that the order of every almost simple image of $G$ is divisible by at most four different primes. Assume moreover that if an almost simple image of $G$ appears in the right column of the table in Theorem \ref{MainTheorem}, then $G$ contains elements of the orders given in parentheses. Then the Prime Graph Question has an affirmative answer for $G$.
\end{maincorollary}

Torsion units in the integral group ring of $G = \PSL(2,q)$ have been studied by many authors. Still the Zassenhaus Conjecture has been proven only for a fistful of such $G$. We extended the list of those $G$ for which the Zassenhaus Conjecture is known by four:

\begin{maintheorem}\label{ZCPSL}
Let $G = \PSL(2,q)$ where $q$ is a prime power and $q \leq 25$ or $q \in \{31, 32\}$. Then the Zassenhaus Conjecture holds for $G$.
\end{maintheorem}
\begin{proof}
The Zassenaus Conjecture for $\PSL(2,2) \cong S_3$ has been proven in \cite{HughesPearson}, for $\PSL(2,3) \cong A_4$ in \cite{AllenHobby}. The case $\PSL(2,4) \cong \PSL(2,5) \cong A_5$ is the original source of the HeLP method in \cite{LutharPassi}. The case $q = 7$ was first handled in \cite[Example 3.6]{HertweckColloq} before being proven again together with the cases $q = 11$ and $q = 13$ in Hertweck's extension of the Luthar-Passi method \cite{HertweckBrauer}. The group $\PSL(2,9) \cong A_6$ was considered in \cite{HertweckA6} and for $q = 8$ and $q = 17$ the Zassenhaus Conjecture was proven independently in \cite{Gildea} and \cite[Theorem~3.1]{KonovalovKimmiComp}. Finally the cases $q = 19$ and $q = 23$ were investigated by the authors in \cite[Theorem~1.1]{Gitter}.

It thus remains to consider the cases $q \in \{16, 25, 31, 32 \}$. Note that by Lemma \ref{prop:basic} (1) the exponents of $G$ and $\V(\ZZ G)$ coincide. For $q \in \{16, 32\}$ an application of the GAP-package implementing the HeLP method \cite{HeLPPackage}, \cite{HeLPPaper} shows that we only have to exclude the existence of units of order $6$ in $\mathrm{V}(\mathbb{Z}G)$ and this is provided by Theorem \ref{MainTheorem}. For $q \in \{ 25, 31\}$ an application of the package implementing the HeLP method gives the result, but this is also obtainable using the literature: It follows from \cite[Propositions 6.2, 6.7]{HertweckBrauer} and \cite[Theorem 3.5]{HeLPTeil} that the orders of torsion units in $\mathrm{V}(\mathbb{Z}G)$ and the orders of elements in $G$ coincide. For $q = 25$ the Zassenhaus Conjecture then follows from \cite[Propositions 6.1, 6.4, 6.6]{HertweckBrauer} and \cite[Lemma~3.2]{Gitter}. For $q = 31$ a unit of an order different from $15$ is rationally conjugate to an element of $G$ by \cite[Proposition 6.1]{HertweckBrauer} and \cite[Theorem 1]{SylowPSL} and units of order $15$ are handled in \cite[Lemma 2.5.3 c)]{Doktor}.
\end{proof}

\begin{remark}
As already mentioned in the concluding remark of \cite{HertweckHoefertKimmi} it is not known whether units of order $3$ in $\mathrm{V}(\mathbb{Z}\PSL(2,27))$ are rationally conjugate to elements of the group base and it is also not known whether $\mathrm{V}(\mathbb{Z}\PSL(2,27))$ contains units of order $6$. These questions are still open nowadays. The generic result \cite[Theorem 2.3]{AngelMariano} shows that the HeLP method is not sufficient to prove that units of order $14$ in $\mathrm{V}(\mathbb{Z}\PSL(2,29))$ are rationally conjugate to elements of the group base.
\end{remark}

Finally, our last main result gives a local positive answer for the Prime Graph Question in a situation where the modular representation theory of $G$ is well behaved.

\begin{maintheorem}\label{Prime3}
Let $G$ be a finite group such that the order $G$ is not divisible by $9$ and let $q$ be a prime. Then $\mathrm{V}(\mathbb{Z}G)$ contains a unit of order $3q$ if and only if $G$ contains an element of order $3q$.
\end{maintheorem}

The proof of Theorem~\ref{Prime3} can be found at the end of Section~\ref{sec:Generic}.

\section{Lattice Method}

In this paragraph we describe and extend the method introduced in \cite{Gitter}. We start with the technical ingredients.
Let $x$ be an element of a finite group $G$, denote by $x^G$ its conjugacy class in $G$ and let $u = \sum\limits_{g \in G} z_g g \in \ZZ G$. Then 
\[ \varepsilon_x(u) = \sum_{g \in x^G} z_g \]
is called the \emph{partial augmentation of $u$ with respect to $x$} (or $x^G$). Sometimes we denote it also by $\varepsilon_{x^G}(u)$. Note that for a normalized unit $u$ we have $\sum_{x^G}\varepsilon_{x}(u) = 1$, summing over all conjugacy classes of $G$. The connection between rational conjugacy and partial augmentations is provided by \cite[Theorem 2.5]{MRSW}: A unit $u \in \V(\ZZ G)$ of order $n$ is rationally conjugate to an element of $G$ if and only if $\varepsilon_{x}(u^d) \geq 0$ for all $x \in G$ and all divisors $d$ of $n$. For torsion units in $\mathbb{Z}G$ the following is known in general.

\begin{lemma}\label{prop:basic} Let $u \in \mathrm{V}(\mathbb{Z}G)$ be a torsion unit.
\begin{enumerate}
\item The order of $u$ divides the exponent of $G$ \cite[Corollary 4.1]{CohnLivingstone}.
\item $\varepsilon_1(u) = 0$ for $u \neq 1$ (Berman-Higman) \cite[Proposition 1.5.1]{EricAngel1}.
\item If the order of $x$ does not divide the order of $u$, then $\varepsilon_x(u) = 0$ \cite[Theorem 2.3]{HertweckBrauer}.
\end{enumerate}
\end{lemma}

In view of the criterion for rational conjugacy of torsion units in $\V(\ZZ G)$ to elements of $G$, it is important to not only know the partial augmentations of a torsion unit $u$, but also those of its powers $u^d$, where $d$ is a divisor of the order of $u$. For that reason we consider as the \emph{possible partial augmentations for elements of order $n$} the partial augmentations of all $u^d$, for all divisors $d$ of $n$. Since $u^n =1$, the partial augmentations of $u^n$ are not included here. Say e.g. a group possesses two conjugacy classes of involutions, with representatives $x$ and $y$, and one of elements of order $3$, represented by $z$, and none of elements of order $6$. Then a typical tuple of possible partial augmentations of a unit $u \in \V(\ZZ G)$ of order $6$ looks like 
\[ (\varepsilon_{x}(u^3), \varepsilon_{y}(u^3), \varepsilon_{z}(u^2), \varepsilon_{x}(u), \varepsilon_{y}(u), \varepsilon_{z}(u)). \]
As above, we always list only those partial augmentations of units which might not be equal to $0$ by Lemma \ref{prop:basic}.

When one is interested in studying units of order $n$ in the normalized units of the integral group ring of some group $G$, the HeLP-method provides strong restrictions on the partial augmentations of these units. In particular only finitely many possibilities for partial augmentations for units of order $n$ remain. Knowing the partial augmentation of a unit $u \in \mathrm{V}(\mathbb{Z}G)$ and all its powers is equivalent to knowing the eigenvalues with multiplicities of $D(u)$ for all complex representations of $G$. This follows via discrete Fourier inversion from the fact that the character table of a group is an invertible matrix. See \cite[Section 2]{HeLPTeil} for details.

Let $p$ be a prime and $x_1, \ldots ,x_n$ be representatives of the conjugacy classes of elements of order $p$ in $G$. For $u \in \V(\ZZ G)$ set $\tilde{\varepsilon}_p(u) = \sum_{j = 1}^n\varepsilon_{x_j}(u)$. For a conjugacy class $K$ of $G$ denote by $K^p$ the conjugacy class containing the $p$-th powers of the elements of $K$. One fact that comes in handy quite often is the following well known result (cf.\ \cite[Proposition 3.1]{HeLPPaper} for a proof).

\begin{lemma}\label{prop:Wagner}
Let $G$ be a finite group, $u \in \V(\ZZ G)$ a torsion unit and $p$ a prime. Then for every conjugacy class $C$ of $G$
\[ \varepsilon_{C}(u^p) \equiv \sum_{K \, : \, K^p = C} \varepsilon_K(u) \mod p.  \]
Assume $u$ is of order $pq$, where $p$ and $q$ are different primes, and $G$ contains no elements of order $pq$. Then using Lemma \ref{prop:basic} (2) we obtain 
\[ \tilde{\varepsilon}_p(u) \equiv 0 \mod p \ \ \ \ \ {\text{and}} \ \ \ \ \  \tilde{\varepsilon}_p(u) \equiv 1 \mod q.  \]
\end{lemma}

So assume that for some unit $u \in \mathrm{V}(\mathbb{Z}G)$ of order $n$ one knows the eigenvalues $D(u)$ for several complex representations $D$ of $G$. The idea of the lattice method introduced in \cite{Gitter} is to use these eigenvalues to obtain information on the structure of modular $kG$-modules when viewed as $k\langle u \rangle$-modules. We will recall this method and add many additional pieces. For a more detailed account of the idea of the method see \cite{Gitter}. \\

\textbf{Notation.} Let $p$ be an odd prime. Denote by $\mathbb{Z}_p$ the $p$-adic integers, by $\mathbb{Q}_p$ the field of $p$-adic numbers. Further denote by $K$ a finite extension of $\mathbb{Q}_p$, by $R$ the ring of integers of $K$ and by $k$ the residue class field of $R$. Denote by $P$ the maximal ideal of $R$ and by $\bar{.}$ the reduction modulo $P$ for both $R$ and $R$-modules.\\

The method is based on the following facts which are consequences of well known results of integral and modular representation theory. The first lemma is a consequence of \cite[Chapter VII, Theorem 5.3, Theorem 5.5]{Huppert2}.

\begin{lemma}\label{ModularCyclicModules} Let $U = \langle u \rangle$ be a cyclic group of order $p^am$ such that the prime $p$ does not divide $m$ and let $k$ be a field of characteristic $p$ containing a primitive $m$-th root of unity $\zeta$. The $m$ simple non-isomorphic $kU$-modules are one-dimensional as $k$-vector spaces and $u^{p^a}$ acts on them as $\zeta^i$ for $1 \leq i \leq m$ and the $i$ determines the isomorphism type of the module since $u^m$ acts trivially.

Up to isomorphism every indecomposable $kU$-module has only one composition factor and is uniserial. Up to isomorphism there is exactly one indecomposable module of every $k$-dimension between $1$ and $p^a$.
\end{lemma}

Lifting idempotents \cite[Theorem 30.4]{CR1} we obtain:

\begin{proposition}\label{LatticeProposition}
Let $U = \langle u \rangle$ be a cyclic group of order $p^am$, where $p$ does not divide $m$. Let $R$ be a complete local ring containing a primitive $m$-th root of unity $\zeta$. Let $D$ be an $R$-representation of $U$ and let $L$ be an $RU$-lattice affording this representation.\\
Let $A_i$ be sets with multiplicities of $p^a$-th roots of unity such that $\zeta A_1 \cup \zeta^2 A_2 \cup ... \cup \zeta^mA_m$ are the eigenvalues of $D(u)$ in the algebraic closure of $K$ where $A_i = \varnothing$ is possible. Let $V_1, \ ..., \ V_m$ be $KU$-modules such that if $E_i$ is a representation of $U$ affording $V_i$, the eigenvalues of $E_i(u)$ are $\zeta^iA_i.$ Then 
\[L \cong L^{\zeta^1} \oplus ... \oplus L^{\zeta^m} \quad \text{and} \quad \bar{L} \cong \bar{L}^{\zeta^1} \oplus ... \oplus \bar{L}^{\zeta^m}\]
such that $\operatorname{rank}_R(L^{\zeta^i})=\dim_k(\bar{L}^{\zeta^i}) = |A_i|.$ (The superscripts $\zeta^i$ are merely meant as indices.) Moreover $K \otimes_R L^{\zeta^i} \cong V_i$ and the only composition factor of $\bar {L}^{\zeta^i}$ is the $i$-th simple module described in Lemma \ref{ModularCyclicModules}.
\end{proposition}

We collect important facts relevant for the method in the following remarks.
\begin{remark}\label{LatticeRemarks} Keep the assumptions of Proposition \ref{LatticeProposition}.
\begin{itemize}
\item[a)] For any finite group $G$ the Krull-Schmidt-Azumaya-Theorem holds for finitely generated $kG$-modules and finitely generated $RG$-lattices \cite[Theorem 6.12]{CR1}.
\item[b)] To impose stronger restrictions on the isomorphism types of direct summands of $L^{\zeta^i}$ one has to study $R\langle u^m \rangle$-lattices. The representation type of $R\langle u^m \rangle$ may be finite, tame or wild. A complete classification is obtained in \cite{Dieterich}. Roughly speaking the representation type of $R\langle u^m \rangle$ gets more complicated when $a$ or the ramification index of $p$ in $R$ is growing. Since in our application we cannot influence $a$ we are interested in keeping the ramification index of $p$ in $R$ as small as possible.
\item[c)] Let $G$ be a finite group and $\chi$ an ordinary character of $G$. Denote by $\mathbb{Q}(\chi)$ the minimal field extension of $\mathbb{Q}$ containing all the character values of $\chi$. Then by a theorem of Fong \cite[Corollary 10.13]{Isaacs} there is a complex root of unity $\xi$ of order comprime with $p$ such that a representation affording $\chi$ can be realized over $\mathbb{Q}(\chi)(\xi)$. Denote by $\mathbb{Q}(\chi)_p$ the $p$-adic completion of $\mathbb{Q}(\chi)$. So by \cite[II, Theorem 7.12]{Neukirch} there is a finite unramified extension $L$ of $\mathbb{Q}(\chi)_p$ such that a representation affording $\chi$ may be realized over $L$.
\item[d)] As $R$ is a principal ideal domain any $K$-representation is equivalent to an $R$-representation \cite[Proposition 23.16]{CR1}.    
\end{itemize}
\end{remark}

The easiest situation which may appear for $RU$-lattices is described in the following lemma which is a consequence of \cite[Corollary 1 to Theorem 2.2]{Gudivok}.
\begin{lemma}\label{prop:rank_op_RCp_lattices}
Let $U = \langle u \rangle$ be a cyclic group of order $p$, assume that $p$ is unramified in $R$ and denote by $\xi$ a primitive $p$-th root of unity. Then up to isomorphism there are exactly three indecomposable $RU$-lattices $L_1, L_2, L_3$. Let $D_i$ be a $K(\xi)$-representation affording $K(\xi) \otimes_R L_i$. Then:
\begin{itemize}
\item $\operatorname{rank}_R(L_1) = 1$ and the eigenvalue of $D_1(u)$ is $1$.
\item $\operatorname{rank}_R(L_2) = p - 1$ and the eigenvalues of $D_2(u)$ are $\xi, \xi^2,...,\xi^{p-1}$.
\item $\operatorname{rank}_R(L_3) = p$ and the eigenvalues of $D_3(u)$ are $1, \xi, \xi^2,...,\xi^{p-1}$.
\end{itemize} 
\end{lemma}

\begin{definition}\label{partitions}
Let $U$ be a cyclic group of order $p^a$. By Lemma \ref{ModularCyclicModules} there are, up to isomorphism, exactly $p^a$ different indecomposable $kU$-modules $I_1,I_2,...,I_{p^a}$ such that $\dim_kI_\ell = \ell$ for $1 \leq \ell \leq p^a$. Assume that for some $kU$-module $M$ we have $M \cong \mu_{p^a}I_{p^a} \oplus ... \oplus \mu_2I_2 \oplus \mu_1I_1$ where $\mu_i \geq 0$. Let $\mu$ be a partition containing $\mu_{p^a}$ times the number $p^a$,..., $\mu_2$ times $2$ and $\mu_1$ times $1$, i.e.
\[\mu = (\underbrace{p^a,...,p^a}_{\mu_{p^a}}, \dots, \underbrace{2,...,2}_{\mu_2}, \underbrace{1,...,1}_{\mu_1}). \] 
Then we call $\mu$ the \emph{partition associated to $M$}.
\end{definition}

The next theorem follows from \cite[II]{MacDonald}.
\begin{theorem}\label{LRTheorem}
Let $U = \langle u \rangle$ a cyclic $p$-group of order $p^a$ and $k$ a finite field of characteristic $p$. Moreover let $M, V$ and $Q$ be $kU$-modules with associated partitions $\lambda, \mu$ and $\nu$. Then there exists a submodule $\tilde{V}$ of $M$ isomorphic to $V$ such that $M/\tilde{V} \cong Q$ if and only if the Littlewood-Richardson coefficient $c^\lambda_{\mu, \nu}$ does not vanish.
\end{theorem}

\begin{remark}\label{prop:symmetrie_of_lrc}
We give the combinatorial definition relevant for us and list facts about the Littlewood-Richardson coefficient $c^\lambda_{\mu, \nu}$.
\begin{itemize}
\item[a)] \cite[I, (9.2)]{MacDonald} The Littlewood-Richardson coefficient $c^\lambda_{\mu \nu}$ may be computed combinatorially in the following way: Let $Y'$ be a Young diagram corresponding to the partition $\lambda = (\lambda_1,...,\lambda_n)$, i.e. a diagram containing $\lambda_1$ empty boxes in the first row, $\lambda_2$ empty boxes in the second row etc. Let $\mu = (\mu_1,...,\mu_m)$ be a subpartition of $\lambda$, i.e. $m \leq n$ and $\mu_i \leq \lambda_i$ for $1 \leq i \leq m$. Then the skew diagram $Y$ corresponding to $\lambda / \mu$ is obtained from $Y'$ by deleting the left $\mu_1$ boxes in the first row, the left $\mu_2$ boxes in the second row etc. 

By putting entries, in our case positive integers, into the boxes of $Y$ it becomes a skew tableau. A skew tableau is called semistandard, if reading a row from left to right the entries do not decrease and reading a column from top to bottom the entries do always increase. A word $w$ of length $r$ in the alphabet of positive integers has the lattice property if for any $s \leq r$ the first $s$ letters of $w$ contain at least as many $1$'s as $2$'s, as many $2$'s as $3$'s etc. Let $\nu = (\nu_1,...,\nu_k)$. Then $c^\lambda_{\mu \nu}$ is the number of semistandard skew tableaus $Y$ of the form $\lambda/\mu$ containing $\nu_1$ times the entry $1$, $\nu_2$ times the entry $2$ etc. such that the word obtained by reading the rows of $Y$ from right to left and from top to bottom has the lattice property.
\item[b)] By definition, the Littlewood-Richardson coefficient $c^\lambda_{\mu \nu}$ is symmetric in $\mu$ and $\nu$, i.e. $c^\lambda_{\mu \nu} = c^\lambda_{\nu \mu}$ \cite[I, 9]{MacDonald}. In the notation of Theorem \ref{LRTheorem} this implies that $M$ has a submodule $\tilde{V}$ isomorphic to $V$ such that $M/\tilde{V} \cong Q$ if and only if $M$ has a submodule $\tilde{Q}$ isomorphic to $Q$ such that $M/\tilde{Q} \cong V$.
\end{itemize}
\end{remark}

\begin{proof}[Proof of Theorem \ref{LRTheorem}]
The idea is to reformulate the question for $M$ into an equivalent question for a finite module of a discrete valuation ring. Once this is done the result follows from \cite[II, 4.3]{MacDonald}.\\
Let $\lambda = (\lambda_1,...,\lambda_s)$. Let $J_i$ denote a Jordan block of size $i$ over $k$ for the eigenvalue $1$. Note that the order of such a Jordan block is $p^\alpha$ where $p^{\alpha - 1} < i \leq p^\alpha$. Then $M$ may be viewed as an $k$ vector space of dimension $\lambda_1 + ... + \lambda_s$ on which $u$ acts as $A = {\text{diag}}(J_{\lambda_1},...,J_{\lambda_s})$. Let $\xi$ be a primitive $p^{a+1}$-th root of unity and set $\mathfrak{o} = \mathbb{Z}_p[\xi]$. Let $\pi: \mathbb{Z}_p \rightarrow k$ be the natural projection. We can equip $M$ with the structure of an finite $\mathfrak{o}$-module by letting an element $x \in \mathbb{Z}_p$ act as scalar multiplication by $\pi(x)$ and $\xi$ act by multiplication by $A$. That is possible since $A$ satisfies the projection of the minimal polynomial of $\xi$ over $\mathbb{Z}_p$ onto $k$, this projection equals $(X-1)^{p^a(p-1)}$ in $k[X]$. This $\mathfrak{o}$-module clearly has a submodule isomorphic to $V$ whose quotient is isomorphic to $Q$ if and only if $M$ as a $kU$-module has such a submodule. Thus the theorem follows from \cite[II, 4.3]{MacDonald}. 
\end{proof}

\section{Generic Results}\label{sec:Generic}
We fix some notation for this section. $p$ will denote an odd prime and $q$ a prime different from $p$. For a character $\chi$ of $G$ we denote by $\chi'$ the restriction of $\chi$ to the $p$-regular classes of $G$, i.e. the $p$-Brauer character corresponding to $\chi$. Denote by $\mathbb{Z}_p[\chi]$ the smallest extension of $\mathbb{Z}_p$ containing all character values of $\chi$. We denote by $\text{Irr}(G)$ the complex irreducible characters of $G$. By $\mathbf{1}$ we denote the trivial character of $G$.

For a local ring $R$ we denote by $\bar{.}$ the reduction modulo the maximal ideal of $R$ for both $R$ and lattices of group rings over $R$. We set $k = \bar{R}$.

\begin{lemma}\label{prop:AltesOmnibusLemma}
Let $\chi, \psi \in \operatorname{Irr}(G)$ such that $\psi' = \mathbf{1}' + \chi'$ and  $\chi'$ is an irreducible Brauer character. Moreover assume that $\mathbb{Z}_p[\chi]$ and $\mathbb{Z}_p[\psi]$ are unramified over $\mathbb{Z}_p$. Then for $u \in \mathrm{V}(\mathbb{Z}G)$ of order $pq$ we have
\[\chi(u) - \psi(u) = \chi(u^q) - \psi(u^q). \]
\end{lemma}

\begin{proof} Let $R$ be an unramified extension of $\mathbb{Z}_p$, containing a primitive $q$-th root of unity $\zeta_q$, over which the characters $\chi$ and $\psi$ may be realized. Such an $R$ exists by Remark \ref{LatticeRemarks}. Let $D_\chi$ and $D_\psi$ be $R$-representations of $G$ affording $\chi$ and $\psi$, respectively, and let $L_\chi$ and $L_\psi$ be corresponding $RG$-lattices. Then, by assumption, the composition factors of $\bar{L}_\psi$ as $kG$-module are $\bar{L}_{\chi}$ and the trivial module, each with multiplicity one.  

Then, by Proposition \ref{LatticeProposition}, as $k\langle u \rangle$-modules we have 
\[\bar{L}_\chi \cong \bar{L}_\chi^{\zeta_1} \oplus \bar{L}_\chi^{\zeta_q^1} \oplus ... \oplus \bar{L}_\chi^{\zeta_q^{q-1}} \quad \text{and} \quad \bar{L}_\psi \cong \bar{L}_\psi^{\zeta_1} \oplus \bar{L}_\psi^{\zeta_q^1} \oplus ... \oplus \bar{L}_\psi^{\zeta_q^{q-1}}. \]
Since the trivial $kG$-module has to be a submodule or quotient of $\bar{L}_\psi^1$ we thus have 
\begin{align}\label{IsoChiPsi}
\bar{L}_\chi^{\zeta_q^1} \oplus ... \oplus \bar{L}_\chi^{\zeta_q^{q-1}} \cong \bar{L}_\psi^{\zeta_q^1} \oplus ... \oplus \bar{L}_\psi^{\zeta_q^{q-1}}. 
\end{align}
 
Assume that $D_\chi(u) \sim (Y_\chi, X_\chi)$ and $D_\psi(u) \sim (Y_\psi, X_\psi)$, where $Y_\chi$ and $Y_\psi$ are the eigenvalues of $D_\chi(u)$ and $D_\psi(u)$ that are $p$-th roots of unity and $X_\chi$ and $X_\psi$ are the eigenvalues not being $p$-th roots of unity (i.e.\ primitive $pq$-th and primitive $q$-th roots of unity).
Since $p$ is unramified in $R$ by Lemma \ref{prop:rank_op_RCp_lattices}, the eigenvalues $X_\chi$ are determined by the isomorphism type of $\bar{L}_\chi^{\zeta_q^1} \oplus ... \oplus \bar{L}_\chi^{\zeta_q^{q-1}}$, the same holds for $\psi$ and thus $X_\chi = X_\psi$ by \eqref{IsoChiPsi}.
Hence 
\begin{equation} \chi(u) = \sum_{x \in X_\chi} x + \sum_{y \in Y_\chi} y \quad \text{and}\quad \psi(u) = \sum_{x \in X_\chi} x + \sum_{y \in Y_\psi} y \label{eq:chiu_psiu}. \end{equation}
As $p$ is unramified in $R$, each primitive $p$-th root of unity is contained in $Y_\chi$ with the same multiplicity, thus $\sum_{y \in Y_\chi} y = \sum_{y \in Y_\chi} y^q$. The same is true when $Y_\chi$ is replaced by $Y_\psi$. So we have 
\begin{equation} \chi(u^q) = \sum_{x \in X_\chi} x^q + \sum_{y \in Y_\chi} y \quad \text{and}\quad \psi(u^q) = \sum_{x \in X_\chi} x^q + \sum_{y \in Y_\psi} y. \label{eq:chiuq_psiuq} \end{equation}
Subtracting the corresponding expressions in \eqref{eq:chiu_psiu} and \eqref{eq:chiuq_psiuq} yields the result.  
\end{proof}

The following describes a situation we frequently encounter in our investigations.
\begin{proposition}\label{prop:Omnibusproposition}
Assume $G$ contains no elements of order $pq$ and $p^2$. Let $\chi, \psi \in \operatorname{Irr}(G)$ such that $\psi' = \mathbf{1}' + \chi'$. Assume moreover that $\chi'$ is irreducible and $\chi$ and $\psi$ are  constant and integral on elements of order $p$. Then $\mathrm{V}(\mathbb{Z}G)$ contains no unit of order $pq$.
\end{proposition}
\begin{proof}
Let $u \in \mathrm{V}(\mathbb{Z}G)$ be a unit of order $pq$ and $g \in G$ some fixed element of order $p$.  Let $q_1, ..., q_k$ be representatives of the conjugacy classes of elements of order $q$ in $G$. Since $\psi(q_i) = \chi(q_i) + 1$ for all $i$ we have
\begin{align*}& \chi(u) = \tilde{\varepsilon}_p(u)\chi(g) + \sum_{i = 1}^k \varepsilon_{q_i}(u)\chi(q_i) \\ \text{and} \quad &\psi(u) = \tilde{\varepsilon}_p(u)\psi(g) + \sum_{i = 1}^k \varepsilon_{q_i}(u)(\chi(q_i) + 1).\end{align*}
Moreover
\[\chi(u^q) - \psi(u^q) = \chi(g) - \psi(g). \]
By the assumptions we can apply Lemma \ref{prop:AltesOmnibusLemma}. Hence, using $\tilde{\varepsilon}_p(u) + \tilde{\varepsilon}_q(u) = 1$, we obtain
\begin{align*}
 \chi(g) - \psi(g) = \tilde{\varepsilon}_p(u)(\chi(g) - \psi(g)) - \tilde{\varepsilon}_q(u) \quad \text{i.e.} \quad \tilde{\varepsilon}_q(u)(\chi(g) - \psi(g) + 1) = 0. 
\end{align*}
If $\tilde{\varepsilon}_q(u) = 0$, then $\tilde{\varepsilon}_p(u) = 1$, contradicting Lemma \ref{prop:Wagner}. Thus
\[\psi(g) = \chi(g) + 1. \]
If $h \in G$ is an element of order $pm$, where $p \nmid m$, the values $\chi(h)$ and $\psi(h)$ are determined by the values on $h^p$ and $h^m$. As there are no elements of order $p^2$ in $G$, we obtain $\psi = \chi + \mathbf{1}$, contradicting the irreducibility of $\psi$.
\end{proof}

Proposition \ref{prop:Omnibusproposition} allows us to prove that if the Sylow $3$-subgroup is of order $3$ then the prime graph of $G$ and $\mathrm{V}(\mathbb{Z}G)$ graphs are equal around the vertex $3$ (Theorem D).
Note that if the Sylow $2$-subgroup of $G$ is of order $2$ then these graphs are even globally equal. This follows from Burnside's $p$-complement Theorem \cite[Kapitel IV, Hauptsatz 2.6]{Huppert1} in conjunction with the Feit-Thompson Theorem and Kimmerle's \cite[Proposition 4.3]{KimmiPQ}.  

 \begin{proof}[Proof of Theorem \ref{Prime3}] Let $P$ be a Sylow $3$-subgroup of $G$, so $P$ is a group of order $3$. By Lemma~\ref{prop:basic} there is no element of order $9$ in $\V(\ZZ G)$, so we can assume $q \not= 3$. 
 
If $N_G(P) = C_G(P)$ then $P$ has a normal $3$-complement in $G$, by Burnside's normal $p$-complement Theorem \cite[Kapitel IV, Hauptsatz 2.6]{Huppert1}. In this case \cite[Proposition~2.2]{KonovalovKimmiComp} shows that there is a normalized unit of order $3q$ in $\ZZ G$ if and only if there is an element of that order in $G$.
 
Now assume $N_G(P) \not= C_G(P)$, so the two non-trivial elements of $P$ are conjugate in $G$ and $[N_G(P):C_G(P)] = 2$. As the principal $3$-block $B_0$ of $G$ has maximal defect, its defect group $D$ is cyclic of order $3$. By Brauer's theory of blocks with defect one, the decomposition matrix of $B_0$ looks as in Table~\ref{tab:decompNavarro} (see \cite[Theorem~11.1]{Navarro}).
\setcounter{table}{0}
 \begin{table}[h]\centering
\caption{Decomposition matrix at $3$ of $B_0$ in case $N_G(P) \neq C_G(P)$}\label{tab:decompNavarro}
\begin{tabular}{l c c }\hline
{ } & $\varphi_{1}$  & $\chi'$ \\ \hline\hline
$\mathbf{1}$ & 1 & $\cdot$ \\
$\chi$ & $\cdot$ & 1 \\
$\psi$ & 1 & 1  \\ \hline
\end{tabular}
\end{table}
In this case $G$ has only one conjugacy class of elements of order $3$ and this class is rational, so all characters of $G$ are integral and constant on elements of order $3$. Recall that $G$ does not contain elements of order $3^2$. So in case there is no element of order $3q$ in $G$, we are in the setting of Proposition \ref{prop:Omnibusproposition} and we conclude that there is no normalized unit of order $3q$ in $\ZZ G$. 
\end{proof}

\section{Applications}
Theorem \ref{Prime3} immediately applies to infinite families relevant for us.
\begin{proposition}
Let $G = \PSL(2,q)$ for some prime power $q \geq 4$ such that $9$ does not divide the order of $G$ and let $p$ be a prime. Then there is a unit of order $3p$ in $\mathrm{V}(\mathbb{Z}G)$ if and only if there is an element of order $3p$ in $G$.
\end{proposition}

\begin{proof} The Sylow $3$-subgroup of $G$ is of order $3$, since the order of $G$ is divisible by $3$ but not by $9$. Hence the claim follows from Theorem \ref{Prime3}.
\end{proof}

Using \cite[Proposition 1.2, Lemma 4.1]{HeLPTeil} this implies the Prime Graph Question for a possibly infinite series of groups. Whether this really is an infinite series is an open number theoretical question \cite[Problem 13.65]{Kourovka}.
\begin{corollary} \label{PQPSL}
Let $G$ be a $4$-primary group isomorphic to some $\PSL(2,2^f)$. Then the Prime Graph Question holds for $G$.
\end{corollary}

\begin{proposition}\label{prop:PSU_3_q_3p}
Let $G = \PSU(3,q)$ for some prime power $q$ such that $3 \mid q - 1$ and $9 \nmid q - 1$. Let $p$ be a prime. Then there is a unit of order $3p$ in $\mathrm{V}(\mathbb{Z}G)$ if and only if there is an element of order $3p$ in $G$.
\end{proposition}

\begin{proof} The order of $G$ is $q^3(q +1)^2(q-1)(q^2 - q + 1)$ \cite[p.\ 564]{GeckPSU}. So the assumption implies that the Sylow $3$-subgroup is of order $3$ and hence the claim follows from Theorem \ref{Prime3}.
\end{proof}

\begin{lemma}\label{prop:PSL_3_5_PSU_3_7} There is no unit of order $15$ in $\ZZ \Aut(\PSL(3,5))$. There is no unit of order $21$ in $\ZZ \Aut(\PSU(3,7))$.
\end{lemma}

\begin{proof} Using \cite{GAP} we see that the Sylow $3$-subgroups of both groups are of order $3$. Hence the claim follows from Theorem \ref{Prime3}.
\end{proof}

\begin{corollary}\label{PQPSL35}
The Prime Graph Question has an affirmative answer for almost simple groups with socle isomorphic to $\PSL(3,5)$, $\PSU(3,4)$ or $\PSU(3,7)$.
\end{corollary}

\begin{proof} By \cite[Theorem B]{HeLPTeil} it only remains to exclude the existence of units of order $6$ in $\V(\ZZ \PSU(3,4))$ and the existence of units of order $15$ and $21$, respectively, in $\V(\ZZ G)$, where $G$ is any automorphic extension of $\PSL(3,5)$ and $\PSU(3,7)$, respectively.  This is done in Proposition \ref{prop:PSU_3_q_3p} and Lemma \ref{prop:PSL_3_5_PSU_3_7}. \end{proof}

We will proceed to handle some more groups, where the application of the method is more involved.\\

\textbf{Notation.} Let $G$ be a finite group. For the rest of this paper we use the following notation. We use characters from the \texttt{GAP} character table library \cite{CTblLib} and use these library names for them. For an ordinary character $\chi_i$ of $G$ we denote by $D_i$ a corresponding representation. Let $p$ be a fixed prime, the exact value of $p$ will always be clear from the context. Then by Remark \ref{LatticeRemarks} we can realize $D_i$ over a $p$-adically complete, discrete valuation ring $R$ which is unramified over $\mathbb{Z}_p[\chi_i]$, the $p$-adic closure of the ring of character values of $\chi_i$. In our application $\mathbb{Z}_p[\chi_i]$ will always be unramified over $\mathbb{Z}_p$.

We denote by $L_i$ an $RG$-lattice corresponding to the representation $D_i$ and by $\bar{.}$ the reduction modulo the maximal ideal of $R$ for both $R$ and $RG$-lattices. Moreover we denote by $k$ a finite field of characteristic $p$ containing the residue class field of all complete valuation rings involved and affording all necessary $p$-modular representations of $G$. For a $p$-Brauer character $\varphi_i$ of $G$ we denote by $S_i$ a corresponding $kG$-module.

Our calculations will always involve a torsion unit $u \in \mathrm{V}(\mathbb{Z}G)$. Unless stated otherwise $L_i$ is always viewed as an $R\langle u \rangle$-lattice and $S_i$ and $\bar{L}_i$ as $k\langle \bar{u} \rangle$-modules. We will moreover use the notation introduced in Proposition \ref{LatticeProposition}.

We use the following notation to indicate that certain multiplicities occur with greater multiplicity. For a diagonalizable matrix $A$ we write 
\[A \sim \left(\eigbox{m}{\delta_1, ..., \delta_k}, \eigbox{n}{\eta_1, ..., \eta_\ell}\right)\]
to indicate that the eigenvalues of $A$ are $\delta_1, ..., \delta_k$ (each with multiplicity $m$) and $\eta_1, ..., \eta_\ell$ (each with multiplicity $n$). Those boxes with eigenvalues are used to group eigenvalues ``of the same kind'' together.

\begin{lemma}
Let $G$ be an almost simple group whose socle is isomorphic to $\PSL(2,81)$. Then $\mathrm{V}(\mathbb{Z}G)$ contains no units of order $15$. 
\end{lemma}
\begin{proof} Let $A$ be the automorphism group of $\PSL(2,81)$. Once we show that $\mathrm{V}(\mathbb{Z}A)$ contains no unit of order $15$, the lemma is proven. Let $u \in \V(\ZZ A)$ be of order $15$. Using the \texttt{GAP}-package implementing the HeLP-method \cite{HeLPPackage} we obtain that $u$ has the partial augmentations  
\[ (\varepsilon_{\cc{3a}}(u^5), \varepsilon_{\cc{5a}}(u^3), \varepsilon_{\cc{3a}}(u), \varepsilon_{\cc{5a}}(u)) = (1, 1, 6, -5). \] 
Note that the Brauer table and the decomposition matrix of $G$ for the prime $5$ are not available in \texttt{GAP}. We can determine those parts relevant to us from \cite[Theorem~15.18]{Isaacs}, or rather its implementation in \texttt{GAP} \cite{GAP} via  \texttt{PrimeBlocks(CharacterTable("L2(81).(2x4)"), 5)}. These parts are given in Table \ref{CT_DM_PGammaL_2_81}.

\begin{table}[h]
\caption{Part of the character table and decomposition matrix for $5$ of $\PSL(2,81).(2 \times 4)$}\label{CT_DM_PGammaL_2_81}\centering
\begin{subtable}[c]{0.3\textwidth} \centering
\subcaption{Part of the character table}
\begin{tabular}{l c c c c} \hline
{ } & $\cc{1a}$ & $\cc{3a}$ & $\cc{5a}$ \\ \hline \hline
$\chi_{1}$ & $1$ & $1$ & $1$  \\
$\chi_4$ & $1$ & $1$ & $1$  \\
$\chi_{28}$ & $81$ & $0$ & $1$  \\
$\chi_{29}$ & $81$ & $0$ & $1$  \\
$\chi_{35}$ & $164$ & $2$ & $-1$ \\ \hline
\end{tabular}
\end{subtable}
\qquad
\begin{subtable}[c]{0.42\textwidth} \centering
\subcaption{Part of the decomposition matrix for $5$}
\begin{tabular}{l c c c c} \hline
{ } & $\psi_{1}$  & $\psi_4$ & $\psi_{28}$ & $\psi_{29}$ \\ \hline\hline
$\chi_{1}$ & 1 & $\cdot$ & $\cdot$ & $\cdot$ \\
$\chi_4$ & $\cdot$ & 1 & $\cdot$ & $\cdot$ \\
$\chi_{28}$ & $\cdot$ & $\cdot$ & 1 & $\cdot$ \\
$\chi_{29}$ & $\cdot$ & $\cdot$ & $\cdot$ & 1 \\
$\chi_{35}$ & 1 & 1 & 1 & 1 \\ \hline
\end{tabular}
\end{subtable}
\end{table}

For the multiplicities of the eigenvalues of $u$ under the corresponding representations we obtain
{\small
\begin{align*}
D_1(u) \sim D_4(u) &\sim (1), \\
D_{28}(u) \sim D_{29}(u) &\sim \left(\eigbox{3}{1},\ \eigbox{6}{\zeta_5,...,\zeta_5^{-1}}, \eigbox{7}{\zeta_3,\zeta_3^{-1}},\ \eigbox{5}{\zeta_{15},...,\zeta_{15}^{-1}}\right), \\
D_{35}(u) &\sim \left(\eigbox{20}{1},\ \eigbox{9}{\zeta_5,...,\zeta_5^{-1}}, \eigbox{6}{\zeta_3,\zeta_3^{-1}},\ \eigbox{12}{\zeta_{15},...,\zeta_{15}^{-1}}\right). 
\end{align*}}
We give an example how these kind of eigenvalues can be obtained. Since for $u^5$ all but one partial augmentations vanish, we have $\chi_{28}(u^5) = \chi_{28}(\cc{3a}) = 0$. Similarly $\chi_{28}(u^3) = \chi_{28}(\cc{5a}) = 1$. This implies
\begin{align*}
D_{28}(u^5) &\sim \left( \eigbox{27}{1},\ \eigbox{27}{\zeta_3, \zeta_3^{-1}} \right), \\
D_{28}(u^3) &\sim \left( \eigbox{17}{1},\ \eigbox{16}{\zeta_5, \zeta_5^2, \zeta_5^{-2},\zeta_5^{-1}} \right).
\end{align*}
Now 
\[\chi_{28}(u) = \varepsilon_{\cc{3a}}(u)\chi_{28}(\cc{3a}) + \varepsilon_{\cc{5a}}(u)\chi_{28}(\cc{5a}) = 6 \cdot 0 - 5 \cdot 1 = -5. \]
On the other hand the eigenvalues of $D_{28}(u)$ are pairwise products of the eigenvalues of $D_{28}(u^3)$ and $D_{28}(u^5)$ and the only possibility to sum these products to $-5$ is the one given above. 

Since $\bar{L}_1$ and $\bar{L}_4$ are sub- or factor modules of $\bar{L}_{35}^1$ we may assume by Remark \ref{prop:symmetrie_of_lrc} that $\bar{L}_{28}^{\zeta^3}$ is a submodule of $\bar{L}_{35}^{\zeta^3}$ such that $\bar{L}_{35}^{\zeta^3}/\bar{L}_{28}^{\zeta^3} \cong \bar{L}_{29}^{\zeta^3}$. We will argue using the combinatorial description of Littlewood-Richardson coefficients introduced in Remark \ref{prop:symmetrie_of_lrc}. Let $\lambda$ be a partition corresponding to $\bar{L}_{35}^{\zeta^3}$ in the sense of Definition \ref{partitions} and let $\mu$ be a partition corresponding to $\bar{L}_{28}^{\zeta^3}$. Then by Lemma \ref{prop:rank_op_RCp_lattices} and the eigenvalues computed above $\lambda$ contains exactly twelve entries being equal to $4$ or $5$, while $\mu$ contains exactly five such entries. That means that the skew diagram corresponding to $\lambda/\mu$ has the form given in Figure \ref{YoungDiagramPSL281} Here a white box is appearing independently of the concrete isomorphism types of $\bar{L}_{35}^{\zeta^3}$ and $\bar{L}_{28}^{\zeta^3}$, while a red box might appear or not.

\begin{figure}[h]
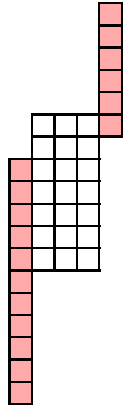

\centering
\ytableausetup{smalltableaux}
\begin{ytableau}
\none & \none & \none & \none & *(Schweinchenrosa) \\
\none & \none & \none & \none & *(Schweinchenrosa)\\
\none & \none & \none & \none & *(Schweinchenrosa) \\
\none & \none & \none & \none & *(Schweinchenrosa) \\
\none & \none & \none & \none & *(Schweinchenrosa) \\
\none & & & & *(Schweinchenrosa) \\
\none & & & \\
*(Schweinchenrosa) & & & \\
*(Schweinchenrosa) & & & \\
*(Schweinchenrosa) & & & \\
*(Schweinchenrosa) & & & \\
*(Schweinchenrosa) & & & \\
*(Schweinchenrosa) \\
*(Schweinchenrosa) \\
*(Schweinchenrosa) \\
*(Schweinchenrosa) \\
*(Schweinchenrosa) \\
*(Schweinchenrosa)  
\end{ytableau}
\caption{Skew diagram corresponding to $\bar{L}_{35}^{\zeta_3}/\bar{L}_{28}^{\zeta_3}$}
\label{YoungDiagramPSL281}
\end{figure}

If one wants to fill this skew diagram with positive integers such that it becomes a semistandard skew tableau satisfying the lattice property some entries are prescribed in any case. Namely the second and third column have to contain the numbers $1$ to $7$, cf. Figure \ref{YoungTableauPSL281}.

\begin{figure}[h]
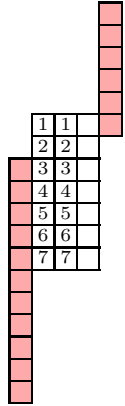

\centering
\ytableausetup{smalltableaux}
\begin{ytableau}
\none & \none & \none & \none & *(Schweinchenrosa) \\
\none & \none & \none & \none & *(Schweinchenrosa)\\
\none & \none & \none & \none & *(Schweinchenrosa) \\
\none & \none & \none & \none & *(Schweinchenrosa) \\
\none & \none & \none & \none & *(Schweinchenrosa) \\
\none & 1 & 1 & & *(Schweinchenrosa) \\
\none & 2 & 2 & \\
*(Schweinchenrosa) & 3 & 3 & \\
*(Schweinchenrosa) & 4 & 4 & \\
*(Schweinchenrosa) & 5 & 5 & \\
*(Schweinchenrosa) & 6 & 6 & \\
*(Schweinchenrosa) & 7 & 7 & \\
*(Schweinchenrosa) \\
*(Schweinchenrosa) \\
*(Schweinchenrosa) \\
*(Schweinchenrosa) \\
*(Schweinchenrosa) \\
*(Schweinchenrosa)  
\end{ytableau}
\caption{Skew tableau corresponding to $\bar{L}_{35}^{\zeta_3}/\bar{L}_{28}^{\zeta_3}$ with some definitive entries, if it should become a semistandard tableau satisfying the lattice property}
\label{YoungTableauPSL281}
\end{figure}

These entries imply that the quotient $\bar{L}_{35}^{\zeta_3}/\bar{L}_{28}^{\zeta_3} \cong \bar{L}_{29}^{\zeta_3}$ has at least seven direct indecomposable summands of dimension at least $2$. However by the eigenvalues of $D_{29}(u)$ computed above and Lemma \ref{prop:rank_op_RCp_lattices} it has exactly five indecomposable summands of this form. A contradiction to the existence of $u$.
\end{proof}

\begin{corollary}\label{PQPSL281}
The Prime Graph Question holds for the group $\PSL(2,81).4b$, $\PSL(2,81).2^2$ and $\Aut(\PSL(2,81)) = \PSL(2,81).(2 \times 4)$.
\end{corollary}
\begin{proof}
If $G$ is one of the three groups mentioned in the corollary, then by \cite[Theorem B]{HeLPTeil} we only have to exclude the existence of units of order $15$ in $\mathrm{V}(\mathbb{Z}G)$. This is done in the preceding lemma. 
\end{proof}

The following example exhibits a somehow surprising situation. Our method does not directly apply to $\PSL(3,4)$, but embedding this group into the Mathieu group of degree $22$ allows a successful application to $\PSL(3,4)$.
\begin{lemma}\label{PQPSL34}
Let $G = \PSL(3,4)$. Then the Prime Graph Question holds for $G$.  
\end{lemma}
\begin{proof} By \cite[Theorem B]{HeLPTeil} it only remains to consider units of order $6$ in $\V(\ZZ G)$. $G$ contains one conjugacy class of involutions and one conjugacy class  of elements of order $3$ but no elements of order $6$. If $\V(\ZZ G)$ contains a unit of order $6$, then its cube and its square are conjugate to group elements and (as stated in \cite[Table~6]{HeLPTeil}) using an irreducible character of degree $35$ and one of degree $45$ it is easy to see that its partial augmentations are $(4, -3)$ or $(-2, 3)$ on the classes of involutions and elements of order $3$ of $G$. Also the application of the lattice method to the group $G$ will not eliminate both possibilities. However, $G$ is a maximal subgroup of the Mathieu group $H = M_{22}$ of degree $22$ \cite[Chapter XII, Theorem 1.4]{Huppert3}. The group $H$ also contains unique conjugacy classes of involutions, say $\cc{2a}$, and of elements of order $3$, say $\cc{3a}$. Thus if a unit of order $6$ in $\V(\ZZ G)$ exists, there exists also a unit in $\V(\ZZ H)$ with the corresponding partial augmentations. So assume we have $u \in \V(\ZZ H)$ with
\[(\varepsilon_{\cc{2a}}(u^3), \varepsilon_{\cc{3a}}(u^2), \varepsilon_{\cc{2a}}(u), \varepsilon_{\cc{3a}}(u)) \in \{ (1, 1, 4, -3), (1, 1, -2, 3) \} .\]
It was already shown in \cite[Theorem 1 (iv)]{KonovalovM22} that this possibility can not be ruled out using the HeLP-method. We will use the characters of $H$ given in Table \ref{CT_DM_M22} along with their decomposition behaviour indicated there. By Remark \ref{prop:symmetrie_of_lrc} we may assume that $\bar{L}_9$ contains a submodule isomorphic to $\bar{L}_8$. By $\zeta$ we denote a primitive third root of unity.

\begin{table}[h]
\caption{Part of the character table and decomposition matrix for $3$ of $M_{22}$}\label{CT_DM_M22}\centering
\begin{subtable}[c]{0.3\textwidth} \centering
\subcaption{Part of the character table}
\begin{tabular}{l c c c c} \hline
{ } & $\cc{1a}$ & $\cc{2a}$ & $\cc{3a}$ \\ \hline \hline
$\chi_{2}$ & $21$ & $5$ & $3$  \\
$\chi_{8}$ & $210$ & $2$ & $3$  \\
$\chi_{9}$ & $231$ & 7 & $-3$  \\  \hline
\end{tabular}
\end{subtable}
\qquad
\begin{subtable}[c]{0.3\textwidth} \centering
\subcaption{Part of the decomposition matrix for $3$}
\begin{tabular}{l c c} \hline
{ } & $\varphi_{2}$  & $\varphi_{9}$ \\ \hline\hline
$\chi_{2}$ & 1 & $\cdot$  \\
$\chi_{8}$ & $\cdot$ & 1  \\
$\chi_{9}$ & 1 & 1   \\ \hline
\end{tabular}
\end{subtable}
\end{table} 

\emph{Case 1}: $(\varepsilon_{\cc{2a}}(u^3), \varepsilon_{\cc{3a}}(u^2), \varepsilon_{\cc{2a}}(u), \varepsilon_{\cc{3a}}(u)) = (1, 1, 4, -3)$.\\
It is easy to compute from the partial augmentations that the eigenvalues of $u$ under $D_8$ and $D_9$ are
\begin{align*} D_8(u) & \sim \left( \eigbox{38}{1},\ \eigbox{34}{\zeta, \zeta^2},\ \eigbox{34}{-1},\ \eigbox{35}{-\zeta, -\zeta^2} \right), \\
 D_9(u) & \sim \left( \eigbox{31}{1},\ \eigbox{44}{\zeta, \zeta^2},\ \eigbox{44}{-1},\ \eigbox{34}{-\zeta, -\zeta^2} \right).
\end{align*}
$\bar{L}_8^{-1}$ has $35$ direct indecomposable summands of dimension at least $2$ by Lemma \ref{prop:rank_op_RCp_lattices}, however $\bar{L}_9^{-1}$ has only $34$ such direct summands. This contradicts the fact that $\bar{L}_8^{-1}$ is isomorphic to a submodule of $\bar{L}_9^{-1}$.

\emph{Case 2}: $(\varepsilon_{\cc{2a}}(u^3), \varepsilon_{\cc{3a}}(u^2), \varepsilon_{\cc{2a}}(u), \varepsilon_{\cc{3a}}(u)) = (1, 1, -2, 3)$.\\
The eigenvalues of $D_8(u)$ and $D_9(u)$ can be calculated to be 
\begin{align*} D_8(u) & \sim \left( \eigbox{36}{1},\ \eigbox{35}{\zeta, \zeta^2},\ \eigbox{36}{-1},\ \eigbox{34}{-\zeta, -\zeta^2} \right), \\
 D_9(u) & \sim \left( \eigbox{51}{1},\ \eigbox{34}{\zeta, \zeta^2},\ \eigbox{24}{-1},\ \eigbox{44}{-\zeta, -\zeta^2} \right)
\end{align*}
in this case. By Lemma \ref{prop:rank_op_RCp_lattices}, $\bar{L}_9^{1}$ has $34$ direct indecomposable summands of dimension $2$ or $3$, contradicting the fact that it contains a submodule isomorphic to $\bar{L}_8^{1}$, which has $35$ such direct summands. 
\end{proof}

\begin{lemma}\label{PQPSL37}
Let $H = \PSL(3,7)$ and $G = \PSL(3,7).2$. Then the Prime Graph Question holds for $H$. If there is a unit $u$ of order $21$ in $\mathrm{V}(\mathbb{Z}G)$, then
\begin{align*}& (\varepsilon_{\cc{7a}}(u^3), \varepsilon_{\cc{7b}}(u^3), \varepsilon_{\cc{7c}}(u^3)) = (5,0,-4) \\ \text{and} \quad &(\varepsilon_{\cc{3a}}(u), \varepsilon_{\cc{7a}}(u), \varepsilon_{\cc{7b}}(u), \varepsilon_{\cc{7c}}(u)) = (-6,-1,0,8) \end{align*}
where the ordering of the conjugacy classes is as in the GAP character table library \cite{CTblLib}. 
\end{lemma}
\begin{proof}
Let $u \in \mathrm{V}(\mathbb{Z}G)$ be of order $21$. Applying the HeLP method with the characters given in \cite[Table 6]{HeLPTeil} we obtain that for all possible $134$ distributions of partial augmentations we have $\varepsilon_{\cc{3a}}(u) = -6$. We will use the characters and their decomposition behaviour for $p=3$ indicated in Table \ref{tab:PSL(3.7)_2mod3}.

\begin{table}[h]
\caption{Part of the character table and the decomposition matrix for $p=3$ for $\operatorname{PSL}(3,7).2$}\label{tab:PSL(3.7)_2mod3}
\centering
\begin{subtable}[c]{0.4\textwidth} \centering
\subcaption{Part of the character table}{
\begin{tabular}{l c c c c c}  \hline
{ } & $\cc{1a}$ & $\cc{3a}$ & $\cc{7a}$ & $\cc{7b}$ & $\cc{7c}$ \\ \hline \hline
$\chi_{5}$ & 57 & 3 & 8 & 1 & 1 \\
$\chi_{23}$ & 399 & 3 & 7 & 0 & 0 \\
$\chi_{25}$ & 456 & $-3$ & 15 & 1 & 1 \\  \hline
\end{tabular}}
\end{subtable}
\qquad
\begin{subtable}[c]{0.3\textwidth} \centering
\subcaption{Part of the decomposition matrix for $p = 3$}{
\begin{tabular}{l c c}  \hline
{ } & $\varphi_{5}$  & $\varphi_{21}$ \\ \hline \hline
$\chi_{5}$ & 1 & $\cdot$  \\
$\chi_{23}$ & $\cdot$ & 1 \\
$\chi_{25}$ & 1 & 1 \\ \hline
\end{tabular}}
\end{subtable}
\end{table} 

Let $a_1$, $a_3$, $a_7$, $a_{21}$, $b_1$, $b_3$, $b_7$ and $b_{21}$ be non-negative integers such that
\begin{align*}
D_{5}(u) &\sim \left(\eigbox{a_1}{1},\ \eigbox{a_3}{\zeta_3,\zeta_3^{-1}},\ \eigbox{a_7}{\zeta_7,...,\zeta_7^{-1}},\ \eigbox{a_{21}}{\zeta_{21},...,\zeta_{21}^{-1}}\right), \\
D_{23}(u) &\sim \left(\eigbox{b_1}{1},\ \eigbox{b_3}{\zeta_3,\zeta_3^{-1}},\ \eigbox{b_7}{\zeta_7,...,\zeta_7^{-1}},\ \eigbox{b_{21}}{\zeta_{21},...,\zeta_{21}^{-1}}\right). 
\end{align*}
From these multiplicities of eigenvalues we obtain the equations
\begin{align}
\chi_5(u) &= a_1 + a_{21} - a_3 - a_7, \nonumber \\
\chi_5(u^3) &= a_1 + 2a_3 - a_7 - 2a_{21}, \nonumber \\
\chi_5(u^7) &= a_1 + 6a_7 - a_3 - 6a_{21} = \chi_5(\cc{3a}) = 3,\nonumber \\
\chi_5(u^{21}) &= a_1 +2a_3 + 6a_7 + 12 a_{21} = \chi_5(\cc{1a}) = 57 \label{chi521}  
\end{align}
and analogues equations can be obtained for $\chi_{23}$.
Independent of the exact values of the partial augmentations of $u$ and $u^3$ we have
\begin{align*}
\chi_5(u^3)+\chi_{23}(u^3) &= (8\varepsilon_{\cc{7a}}(u^3)+\varepsilon_{\cc{7b}}(u^3) + \varepsilon_{\cc{7c}}(u^3)) + (7\varepsilon_{\cc{7a}}(u^3)) \\
 &= 15\varepsilon_{\cc{7a}}(u^3) + \varepsilon_{\cc{7b}}(u^3) + \varepsilon_{\cc{7c}}(u^3) =  \chi_{25}(u^3).
\end{align*}
Let $r = 15\varepsilon_{\cc{7a}}(u) + \varepsilon_{\cc{7b}}(u) + \varepsilon_{\cc{7c}}(u).$ Then
\[\chi_5(u)+\chi_{23}(u) = r + 6 \varepsilon_{\cc{3a}}(u) = r - 36 = (r + 18) - 54 = \chi_{25}(u) - 54\]
and this implies

\begin{align*} D_{25}(u) \sim \Big(&\eigbox{a_1 + b_1 + 30}{1},\ \eigbox{a_3 + b_3 - 15}{\zeta_3,\zeta_3^{-1}},\\ & \eigbox{a_7 + b_7 - 6}{\zeta_7,...,\zeta_7^{-1}},\ \eigbox{a_{21} +b_{21} + 3}{\zeta_{21},...,\zeta_{21}^{-1}} \Big).\end{align*}
These multiplicities may be verified  using the character values $\chi_{25}(u),$ $\chi_{25}(u^3),$  $\chi_{25}(u^7)$ and $\chi_{25}(u^{21})$. More precisely we obtain the equalities
\begin{align*}
\chi_{25}(u) &= a_1 + b_1 + 30 + a_{21} + b_{21} + 3 - a_3 - b_3 + 15 - a_7 - b_7 + 6  \\ 
&= a_1 + a_{21} - a_3 - a_7 + b_1 + b_{21} - b_3 - b_7 + 54  \\
&= \chi_5(u) + \chi_{23}(u) +54, \\
\chi_{25}(u^3) &= a_1 + b_1 + 30 + 2(a_3 + b_3 -15) - a_7 - b_7 + 6 -2(a_{21} + b_{21} + 3) \\
  &= a_1 + 2a_3 - a_7 - 2a_{21} + b_1 + 2b_3 - b_7 - 2b_{21} = \chi_5(u^3) + \chi_{23}(u^3),\\
\chi_{25}(u^7) &= a_1 + b_1 + 30 + 6(a_7 + b_7 - 6) - a_3 - b_3 + 15 - 6(a_{21}+b_{21} +3) \\ 
&= a_1 +6a_7 - a_3 - 6a_{21} +b_1 +6b_7 - b_3 - b_{21} - 9  \\
 &= \chi_5(u^7) +\chi_{23}(u^7) - 9 = -3 = \chi_{25}(\cc{3a}), \\
 \chi_{25}(u^{21}) &= a_1 + b_1  + 30 + 2(a_3 + b_3 - 15) + 6(a_7 + b_7 - 6) + 12(a_{21} + b_{21} + 3)  \\
  &= a_1 + 2a_3 + 6a_7  + 12a_{21} + b_1 + 2b_3 + 6b_7 + 12b_{21}  \\
  &= \chi_5(u^{21}) + \chi_{23}(u^{21}) = 57 + 399 = 456 = \chi_{25}(\cc{1a}).     
\end{align*}
These equalities prove the claimed eigenvalues for $D_{25}(u)$. 

We will use the lattice method and in particular Lemma \ref{prop:rank_op_RCp_lattices} several times without further reference. Since $\bar{L}_{23}^1$ is a sub- or factor module of $\bar{L}_{25}^1$ we have
\[a_3 + b_3 - 15 \geq b_3\]
and so $a_3 \geq 15$. Moreover from $\chi_5(u^7) = \chi_5(\cc{3a}) = 3$ we get
\[D_5(u^7) \sim \left(\eigbox{21}{1},\ \eigbox{18}{\zeta_3,\zeta_3^{-1}} \right),\]
implying 
\[a_3+6a_{21} = 18.\]
Since $a_3 \geq 15$, we thus obtain $a_3 = 18$ and $a_{21} = 0$. Hence $\bar{L}_5^{\zeta_7} \cong a_7k$ is a sum of trivial modules and moreover by \eqref{chi521} we have $a_7 \leq 3$. Since $\bar{L}_{25}^{\zeta_7}$ has at least three more indecomposable summands of dimension $2$ or higher compared with $\bar{L}_{23}^{\zeta_7}$ we have
\begin{align*}
3 \geq {\rm{dim}}(\bar{L}_{25}^{\zeta^7}/\bar{L}^{\zeta^7}_{23}) = {\rm{dim}}(\bar{L_5}^{\zeta^7}) = {\rm{dim}}(a_7k) = a_7,
\end{align*}
implying $a_7 = 3$. From \eqref{chi521} we then get $a_1 = 3$. This implies
\[D_5(u^3) \sim \left(\eigbox{39}{1},\ \eigbox{3}{\zeta_7,...,\zeta_7^{-1}} \right)\]
and so
\[\chi_5(u^3) = 36 = 8\varepsilon_{\cc{7a}}(u^3) + \varepsilon_{\cc{7b}}(u^3) + \varepsilon_{\cc{7c}}(u^3).\]
Using $\varepsilon_{\cc{7a}}(u^3) + \varepsilon_{\cc{7b}}(u^3) + \varepsilon_{\cc{7c}}(u^3) = 1$ we obtain $\varepsilon_{\cc{7a}}(u^3) = 5.$ 

Furthermore
\begin{align*}
\chi_5(u) &= a_1 + a_{21} - a_3 - a_7 = 3 - 18 - 3 = -18  \\
  &= 3\varepsilon_{\cc{3a}}(u) + 8\varepsilon_{\cc{7a}}(u) + \varepsilon_{\cc{7b}}(u) + \varepsilon_{7c}(u).
\end{align*}
\hspace{-2.7mm} Together with 
\[\varepsilon_{\cc{7a}}(u) + \varepsilon_{\cc{7b}}(u) + \varepsilon_{\cc{7c}}(u) =  1 - \varepsilon_{\cc{3a}}(u) = 7\]
this implies $\varepsilon_{\cc{7a}}(u)=-1.$ 

Filtering the $134$ possible distributions of partial augmentations for $u$ we started with for these properties we are left with $11$ possibilities all of which satisfy $\varepsilon_{\cc{7c}}(u) \in \{-2,-3,-4\}$. We will use the characters and their decomposition behaviour with respect to $p = 3$ given in Table \ref{tab:PSL(3.7)_2Teil2}.
\begin{table}[h]
\caption{Part of the character table and decomposition matrix with respect to $p=3$ for $\operatorname{PSL}(3,7).2$}\label{tab:PSL(3.7)_2Teil2}\centering
\begin{subtable}[c]{0.4\textwidth} \centering
\subcaption{Part of the character table}{
\begin{tabular}{l c c c c c}  \hline
{ } & $\cc{1a}$ & $\cc{3a}$ & $\cc{7a}$ & $\cc{7b}$ & $\cc{7c}$ \\ \hline  \hline
$\chi_1$ & $1$ & $1$ & $1$ & $1$ & $1$ \\ 
$\chi_2$ & $1$ & $1$ & $1$ & $1$ & $1$ \\ 
$\chi_4$ & 56 & 2 & 7 & 0 & 0 \\
$\chi_7$ & 152 & $-1$ & 5 & 5 & $-2$ \\  \hline
\end{tabular}}
\end{subtable}
\qquad
\begin{subtable}[c]{0.4\textwidth} \centering
\subcaption{Part of the decomposition matrix for $p=3$}{
\begin{tabular}{l c c c c} \hline
{ } & $\varphi_{1}$  & $\varphi_2$ & $\varphi_{3}$ & $\varphi_7$ \\ \hline \hline
$\chi_1$ & $1$ & $\cdot$ & $\cdot$ & $\cdot$ \\
$\chi_2$ & $\cdot$ & $1$ & $\cdot$ & $\cdot$ \\
$\chi_{4}$ & $\cdot$ & 1 & 1 & $\cdot$  \\
$\chi_{7}$ & 1 & $\cdot$ & 1 & 1 \\ \hline
\end{tabular}}
\end{subtable}
\end{table}
From the equation $\chi_4(u) = 2\varepsilon_{\cc{3a}}(u) + 7\varepsilon_{\cc{7a}}(u) = -19$ and the partial augmentations of $u^3$ and $u^7$ we obtain the eigenvalues of $D_4(u)$ as
\[D_4(u) \sim \left(\eigbox{2}{1},\ \eigbox{18}{\zeta_3,\zeta_3^{-1}},\ \eigbox{3}{\zeta_7,...,\zeta_7^{-1}}\right).\]
Let
\[D_7(u) \sim \left(\eigbox{c_1}{1},\ \eigbox{c_3}{\zeta_3,\zeta_3^{-1}},\ \eigbox{c_7}{\zeta_7,...,\zeta_7^{-1}},\ \eigbox{c_{21}}{\zeta_{21},...,\zeta_{21}^{-1}}\right).\]
From the eigenvalues of $D_4(u)$ and the decomposition behaviour of $\chi_4$ we know that $S_3^1$, as a sub- or factor module of $\bar{L}_4^1$, has at least $17$ direct indecomposable summands of $k$-dimension at least $2$. Since $S_3^1$ is also a sub- or factor module of $\bar{L}_7^1$, this implies $c_3 \geq 17.$ Moreover we know $\chi_7(u^7) = -1,$ implying  
\[D_7(u^7) \sim \left(\eigbox{50}{1},\ \eigbox{51}{\zeta_3,\zeta_3^{-1}}\right).\] 
This gives $c_3 + 6c_{21} = 51$ and hence $c_3 \equiv 3 \mod 6.$

We will separate the three possibilities $\varepsilon_{\cc{7c}}(u) \in \{-2,-3,-4\}$.
 
\textit{Case 1:}  $\varepsilon_{\cc{7c}}(u^3) = -2,$ so $(\varepsilon_{\cc{7a}}(u^3), \varepsilon_{\cc{7b}}(u^3), \varepsilon_{\cc{7c}}(u^3)) = (5,-2,-2).$\\
In this case $\chi_7(u^3) = 19$ and 
\[D_7(u^3) \sim \left(\eigbox{38}{1},\ \eigbox{19}{\zeta_7,...,\zeta_7^{-1}}\right).\]
Thus $D_7(u)$ has exactly $6\cdot 19 = 114$ eigenvalues being primitive $7$th or $21$st roots of unity. Since there are only $38$ eigenvalues remaining, we obtain $c_3 \leq 19.$ This contradicts $c_3 \geq 17$ and $c_3 \equiv 3 \mod 6.$

\textit{Case 2:} $\varepsilon_{\cc{7c}}(u^3) = -3,$ so $(\varepsilon_{\cc{7a}}(u^3), \varepsilon_{\cc{7b}}(u^3), \varepsilon_{\cc{7c}}(u^3)) = (5,-1,-3).$\\
The same way as in Case 1 we obtain 
\[D_7(u^3) \sim \left(\eigbox{44}{1},\ \eigbox{18}{\zeta_7,...,\zeta_7^{-1}}\right).\]
So $D_7(u)$ has exactly 108 eigenvalues being primitive $7$th or $21$st roots of unity. This implies $c_3 \leq 22$ and so $c_3=21.$ From the equation $c_3 + 6c_{21} = 51$ we get $c_{21} = 5,$ and so $c_7=8$ and $c_1 = 2.$ Thus 
\[\chi_7(u) = c_1 + c_{21} - c_3 - c_7 = -22.\]
Moreover from $\varepsilon_{\cc{7b}}(u) + \varepsilon_{\cc{7c}}(u) = 8$ and
\begin{align*}
-22 = \chi_7(u) &= -\varepsilon_{\cc{3a}}(u) +5\varepsilon_{\cc{7a}}(u) + 5\varepsilon_{\cc{7b}}(u) - 2\varepsilon_{\cc{7c}}(u) \\
 &= 6 - 5  + 5(8-\varepsilon_{7c}(u)) - 2\varepsilon_{7c}(u) = 41 -7\varepsilon_{7c}(u). 
\end{align*}   
we obtain $\varepsilon_{\cc{7c}}(u) = 9$ and $\varepsilon_{\cc{7b}}(u) = -1.$

We will use the characters and their decomposition behaviour with respect to $p=7$ given in Table \ref{tab:PSL(3.7)_2mod7}. Since from here on we only use those characters, we also denote them by $\varphi_i$.
\begin{table}[h]
\caption{Parts of the ordinary character table, the $7$-modular Brauer character table and the decomposition matrix for $p=7$ for $\operatorname{PSL}(3,7).2$.}\label{tab:PSL(3.7)_2mod7}\centering
\begin{subtable}[c]{0.4\textwidth} \centering
\subcaption{Part of the character table}{
\begin{tabular}{l c c c c c}  \hline
{ } & $\cc{1a}$ & $\cc{3a}$ & $\cc{7a}$ & $\cc{7b}$ & $\cc{7c}$ \\ \hline  \hline
$\chi_{7}$ & 152 & $-1$ & 5 & 5 & $-2$ \\
$\chi_{8}$ & 152 & $-1$ & 5 & 5 & $-2$ \\
$\chi_{9}$ & 304 & $-2$ & 10 & $-4$ & 3 \\  \hline
\end{tabular}}
\end{subtable}\qquad
\begin{subtable}[c]{0.25\textwidth} \centering
\subcaption{Part of the $7$-modular Brauer table}{
\begin{tabular}{l c c}  \hline
{ } & $\cc{1a}$  & $\cc{3a}$ \\ \hline  \hline
$\varphi_{13}$ & 117 & 0  \\
$\varphi_{14}$ & 117 & 0 \\  \hline
\end{tabular}}
\end{subtable}\\[.3cm]
\begin{subtable}[c]{0.6\textwidth} \centering
\subcaption{Part of the decomposition matrix with for $p=7$}{
\begin{tabular}{l c c c c c c}  \hline
{ } & $\varphi_{3}$  & $\varphi_{4}$ & $\varphi_6$ & $\varphi_7$ & $\varphi_{13}$ & $\varphi_{14}$ \\ \hline  \hline
$\chi_{7}$ & $\cdot$ & 1 & $\cdot$ & 1 & 1 & $\cdot$  \\
$\chi_{8}$ & 1 & $\cdot$ & 1 & $\cdot$ & $\cdot$ & 1 \\
$\chi_{9}$ & 1 & 1 & 1 & 1 & 1 & 1 \\  \hline
\end{tabular}}
\end{subtable}
\end{table}\\
From the values of the given Brauer characters we deduce 
\begin{align*}
\Theta_{13}(u^7) \sim \Theta_{14}(u^7) \sim \left(\eigbox{39}{1}, \eigbox{39}{\zeta_3,\zeta_3^{-1}}\right).
\end{align*}
Hence $S_{13}^1$ and $S_{14}^1$ are 39-dimensional $k\langle \bar{u} \rangle$-modules. From the partial augmentations of $u$ we obtain the eigenvalues of the ordinary representations to be
\begin{align*}
D_7(u) \sim D_8(u) \sim \left(\eigbox{2}{1},\ \eigbox{8}{\zeta_7,...,\zeta_7^{-1}},\ \eigbox{21}{\zeta_3 ,\zeta_3^{-1}},\ \eigbox{5}{ \zeta_{21},...,\zeta_{21}^{-1}}\right), \\
D_9(u) \sim \left(\eigbox{46}{1},\ \eigbox{9}{\zeta_7,...,\zeta_7^{-1}},\ \eigbox{18}{\zeta_3 ,\zeta_3^{-1}},\ \eigbox{14}{\zeta_{21},...,\zeta_{21}^{-1}} \right).
\end{align*}
For a module $M$ we denote by ${\rm{soc}}(M)$ the socle of $M$, i.e. the maximal semisimple submodule of $M$. Since $S_{13}^1$ is a sub- or factor module of both $\bar{L}_7^1$ and $\bar{L}_8^1$, the eigenvalues imply 
\[{\rm{dim}}({\rm{soc}}(S_{13}^1)) \leq 10\]
and thus
\[ {\rm{dim}}(S_{13}^1/{\rm{soc}}(S_{13}^1)) \geq 29 \ \ {\mathrm{and}} \ \ {\rm{dim}}(S_{14}^1/{\rm{soc}}(S_{14}^1)) \geq 29.\]
Since both $S_{13}^1$ and $S_{14}^1$ are sub- or factor modules of $\bar{L}_9^1$, we conclude that 
\[{\rm{dim}}(\bar{L}_9^1/{\rm{soc}}(\bar{L}_9^1)) \geq 2\cdot 29 =58.\]
On the other hand this dimension is bounded by the number of primitive $7$th roots of unity appearing as eigenvalues of $D_9(u)$, implying  
\[{\rm{dim}}(\bar{L}_9^1/{\rm{soc}}(\bar{L}_9^1)) \leq 9\cdot 6 = 54,\]
a contradiction.

\textit{Case 3:} $\varepsilon_{\cc{7c}}(u^3) = -4,$ so $(\varepsilon_{\cc{7a}}(u^3), \varepsilon_{\cc{7b}}(u^3), \varepsilon_{\cc{7c}}(u^3)) = (5,0,-4).$\\
Arguing as in the previous cases we have $\chi_7(u^3) = 33$ and so 
\[D_7(u^3) \sim \left(\eigbox{50}{1},\ \eigbox{17}{\zeta_7,...,\zeta_7^{-1}} \right).\]
Thus $D_7(u)$ has exactly $152-17\cdot6 = 50$ eigenvalues being primitive $7$th or $21$st roots of unity. So $c_3 \leq 25$ and proceeding as in Case 2 we have $c_3 = 21, c_{21} = 5, c_7=7$ and $c_1 = 8.$ This means
\[-15 = \chi_7(u) = 41 - 7\varepsilon_{\cc{7c}}(u)\]
and so $\varepsilon_{7c}(u) = 8$ and $\varepsilon_{\cc{7b}}(u) = 0.$ Unfortunately we were not able to apply the lattice method successfully in this case. So we conclude that if $u$ exists, it satisfies 
\begin{align*} & (\varepsilon_{\cc{7a}}(u^3), \varepsilon_{\cc{7b}}(u^3), \varepsilon_{\cc{7c}}(u^3)) = (5,0,-4)\\ \text{and} \quad (& \varepsilon_{\cc{3a}}(u), \varepsilon_{\cc{7a}}(u), \varepsilon_{\cc{7b}}(u), \varepsilon_{\cc{7c}}(u)) = (-6,-1,0,8). \end{align*}

Now assume that $\mathrm{V}(\mathbb{Z}\PSL(3,7))$ contains a unit $u$ of order $21$. Then this unit lies in the integral group rings of the three isomorphic but different automrophic degree $2$ extensions of $H = \PSL(3,7)$. Note that the outer automorphism group of $H$ is isomorphic to the symmetric group of degree $3$. There are four conjugacy classes of elements of order $7$ in $H$, say $7\alpha, 7\beta, 7\gamma, 7\delta.$ In each degree $2$ extension the three different pairs from the classes $7\beta, 7\gamma, 7\delta$ join into one conjugacy class while the third class plays the role of $7b$ from the calculations above. This implies 
\[\varepsilon_{7\alpha}(u^3) = 5 \ \ {\rm{and}}  \ \ \varepsilon_{7\beta}(u^3) = \varepsilon_{7\gamma}(u^3) = \varepsilon_{7\delta}(u^3) = 0,\]
contradicting the fact that $u$ is a normalized torsion unit.
\end{proof}

\section*{Acknowledgements}
We are very thankful to Markus Schmidmeier for providing us with the link between Littlewood-Richardson coefficients and the structure of finite $\mathfrak{o}$-modules.

\bibliographystyle{amsalpha}      
\bibliography{GitterTeil}   

\end{document}